\documentclass[a4paper,12pt]{article}   % RUN  with F1 key!
\usepackage[top=2.5cm,bottom=2.5cm,left=2.5cm,right=2.5cm]{geometry}
\usepackage{amsmath,amssymb,amsfonts,xspace}

\newtheorem{prethm}{{\bf Theorem}}

\newenvironment{theorem}{\begin{prethm}{\hspace{-0.5
               em}{\bf.}}}{\end{prethm}}
               
\newtheorem{prelem}{{\bf Lemma}}

\newtheorem{prelm}{{\bf Lemma}}

\newenvironment{lm}{\begin{prelm}{\hspace{-0.5
               em}{\bf.}}}{\end{prelm}}                              

\newtheorem{prepro}{{\bf Proposition}}

\newtheorem{precor}{{\bf Corollary}}

\newtheorem{preobserv}{{\bf Observation}}

\newenvironment{observation}{\begin{preobserv}{\hspace{-0.5
               em}{\bf.}}}{\end{preobserv}}

\newtheorem{preconj}{{\bf Conjecture}}

\newtheorem{preremark}{{\bf Remark}}

\newtheorem{predef}{{\bf Definition}}

\newtheorem{preproof}{{\bf Proof.}}

\newenvironment{proof}[1]{\begin{preproof}{\rm
               #1}\hfill{$\Box$}}{\end{preproof}}

\newcommand{\sle}{{\rm SLEE}}

\renewcommand{\thefootnote}
\begin{document}
%%%%%%%%%%%%%%%%%%%%%%%%%%%%%%%%%%
%%%%%%%%%%%%%%%%%%%%%%%%%%%%%%%%%%
\baselineskip=0.30in
\vspace*{6cm}
%%%%%%%%%%%%%%%%%%%%%%%%%%%%%%%%%%
\begin{center}
{\Large  On the Randi$\acute{\mbox{c}}$ and degree distance indices of the Mycielskian of a graph}

%%%%%%%%%%%%%%%%%%%%%%%%%%%%%%%%%%
\vspace{5mm}
%%%%%%%%%%%%%%%%%%%%%%%%%%%%%%%%%%
{\large  Ali Behtoei$^*$, Mahdi Anbarloei}

%%%%%%%%%%%%%%%%%%%%%%%%%%%%%%%%%%
\vspace{4mm}
\baselineskip=0.20in
%%%%%%%%%%%%%%%%%%%%%%%%%%%%%%%%%%
{\it Department of Mathematics, Imam Khomeini International University,} \\
{\it  P.O. Box: 34149-16818, Qazvin, Iran} \\
\footnote{* Corresponding author}
\footnote{e-mail: a.behtoei@sci.ikiu.ac.ir, m.anbarloei@sci.ikiu.ac.ir}
%%%%%%%%%%%%%%%%%%%%%%%%%%%%%%%%%%
\thispagestyle{empty}
\vspace{6mm}
%(Received October 29, 2014)
%%%%%%%%%%%%%%%%%%%%%%%%%%%%%%%%%%
\end{center}
%%%%%%%%%%%%%%%%%%%%%%%%%%%%%%%%%%
\begin{abstract}
 In a search for triangle-free graphs with arbitrarily large chromatic numbers, Mycielski developed a graph transformation that transforms a graph  into a new graph  which is called the Mycielskian of that graph. In this paper we provide some sharp bounds for the Randi$\acute{\mbox{c}}$ index of the Mycielskian graphs.
Also, we determine the degree distance index of the Mycielskian of each graph with diameter two. 
\end{abstract}
%%%%%%%%%%%%%%%%%%%%%%%%%%%%%%%%%%
\baselineskip=0.30in
%%%%%%%%%%%%%%%%%%%%%%%%%%%%%%%%%%
\section{Introduction}

Throughout this paper we consider simple graphs, that are finite and undirected graphs without loops or multiple edges.
 Let $G=(V (G),E(G))$ be a connected graph of order $n=|V(G)|$ and of size $m=|E(G)|$.
The distance between two vertices $u$ and $v$ is denoted by  $d_G(u,v)$  and is the length of a shortest path between $u$ and $v$ in $G$.
The diameter of $G$ is $\max\{d_G(u,v):~u,v\in V(G)\}$. It is well known that almost all graphs have diameter two.
When $u$ is a vertex of $G$, then the neighbor of $u$ in $G$ is the set $N_G(u)=\{v:~uv\in E(G)\}$.
The degree of  $u$  is the number of edges adjacent to $u$ and is denoted by $\deg_G(u)$ .
A graph is said to be regular if all of its vertices have the same degree.
%The notations used in this article are mainly taken from the book  ??? 

In a search for triangle-free graphs with arbitrarily large chromatic number, 
Mycielski \cite{Mycielski} developed an interesting graph transformation as follows. 
For a graph $G=(V,E)$, the {\it Mycielskian} of G is the graph $\mu(G)$, or simply $\mu$, with the disjoint union $V\cup X\cup \{x\}$
as its vertex set and $E\cup \{v_ix_j:~v_iv_j\in E\}\cup \{xx_j:~1\leq j\leq n\}$ as its edge set, 
where $V=\{v_1,v_2,...,v_n\}$ and  $X=\{x_1,x_2,...,x_n\}$.
The Mycielskian and generalized Mycielskians have fascinated graph theorists a great deal. This has resulted
in studying several graph parameters of these graphs (see for instance  \cite{Fisher}).

A {\em chemical graph} is a graph whose vertices denote atoms and edges denote bonds between those atoms of any underlying chemical structure. 
A {\it topological index} for a (chemical) graph $G$ is a numerical quantity invariant under
automorphisms of $G$ and it does not depend on the labeling or pictorial
representation of the graph. 
Topological indices and graph invariants based on the distances between vertices of a graph or vertex degrees are widely used for
characterizing molecular graphs, establishing relationships between structure and properties of molecules, predicting
biological activity of chemical compounds, and making their chemical applications. These indices may be used
to derive quantitative structure-property or structure-activity relationships (QSPR/QSAR).

The concept of topological index came from work done by Harold Wiener in 1947 while he was working on boiling point
of paraffin \cite{Wiener}. He named this index as path number. Later on, path number was renamed as {\it Wiener index}  and then theory of
topological index started. The Wiener index of $G$ is defined as
$W(G)=\sum_{\{u,v\}\subseteq V(G)} d_G(u,v)$.
The very first and oldest degree based topological index is {\it Randi$\acute{\mbox{c}}$ index} \cite{Randic} denoted by $R(G)$ and introduced by Milan
Randi$\acute{\mbox{c}}$ in 1975 as
$$R(G)=\sum_{uv\in E(G)}{1\over \sqrt{\deg_G(u)~\deg_G(v)}}.$$

It has been closely correlated with many chemical properties.
The general Randic index was proposed by Bollobás and Erd$\ddot{\mbox{o}}$s \cite{Bollobas} and Amic et al. \cite{Amic} independently, in 1998. Then it has
been extensively studied by both mathematicians and theoretical chemists \cite{MATCH-Randic}, \cite{HuGutman}. 
For a survey of results, we refer to the new book by Li and Gutman \cite{LiGutman}.
An important topological index introduced about forty years ago 
by Ivan Gutman and Trinajsti$\acute{\mbox{c}}$  \cite{GutmanTrinajstic}  is the {\it Zagreb index} or
more precisely {\it first zagreb index} denoted by $M_1(G)$ and was defined as the sum of degrees of end vertices of all edges of G,
$$M_1(G)=\sum_{uv\in E(G)}(\deg_G(u)+\deg_G(v))=\sum_{x\in V(G)}(\deg_G(x))^2.$$

The {\it degree distance} was introduced by Dobrynin and Kochetova \cite{Dobrynin} and Gutman \cite{Schultz-GutmanIndex} as a weighted version of the
Wiener index. The degree distance of G, denoted by DD(G), is defined as follows and it is computed for important families of graphs ( see\cite{MATCH-Degreedistance} for instance):
$$DD(G)=\sum_{\{ u,v \} \subseteq V(G)} d_G(u,v)~(\deg_G (u)+\deg_G(v)).$$

In this paper we provide upper and lower bounds for the Randic index of the Mycielskian graphs.
Also, we determine the degree distance index of the Mycielskian of each graph with diameter two. 
Up to now and as far as we know, these parameters are not determined for the Mycielskian graphs.

\section{The Degree Distance index}

In order to determine the degree distance index of the Mycielskian graph, we need the following observations. 
Throughout this paper we suppose that $G$ is a connected graph, $V(G)=\{v_1,v_2,...,v_n\}$, 
$X=\{ x_1,x_2,...,x_n\}$,  $V(G)\cap X=\emptyset$, $x\notin V(G)\cup X$,
and
$$V(\mu)=V(G)\cup X\cup\{x\}, ~E(\mu)=E(G)\cup \{v_ix_j:~v_iv_j\in E(G)\}\cup\{xx_i:~1\leq i \leq n\}.$$
\begin{observation}  \label{MycielskiDegree}
 Let $\mu$ be the Mycielskian of $G$. Then for each $v\in V(\mu)$ we have
 \begin{eqnarray*}
 \deg_\mu (v)=
 \begin{cases}
 n & v=x  \\
 1+\deg_G(v_i) & v=x_i    \\
 2\deg_G(v_i) & v=v_i.
 \end{cases}
 \end{eqnarray*}
 \end{observation}

%%%%%%%%%%%%%%%%%%%%%%%%%%%%%%%%%%
\begin{observation}  \label{MycielskiDistance}
Distances between the vertices of  the Mycielskian $\mu$ of $G$ are given as follows.
For each $u,v\in V(\mu)$ we have
\begin{eqnarray*}
d_\mu(u,v)=
\left\{   \begin{array}{ll}
1 & u=x,~v=x_i \\
2 & u=x,~v=v_i \\
2 & u=x_i,~v=x_j \\
d_G(v_i,v_j) &  u=v_i,~v=v_j,~d_G(v_i,v_j)\leq 3 \\
4 & u=v_i,~v=v_j,~d_G(v_i,v_j)\geq 4 \\
2 & u=v_i,~v=x_j,~i=j \\
d_G(v_i,v_j) & u=v_i,~v=x_j,~i\neq j,~d_G(v_i,v_j)\leq 2 \\
3 & u=v_i,~v=x_j,~i\neq j,~d_G(v_i,v_j)\geq 3.
\end{array} \right.
\end{eqnarray*}
\end{observation}
%%%%%%%%%%%%%%%%%%%%%%%%%%%%%%%%%%
Note that there are $|E(G)|$ unordered pairs of vertices in $V(G)$ whose distance is 1, 
$$ | \left\{ \{u,v\}\subseteq V(G): ~ d_G(u,v)=1 \right\} | = |E(G)|. $$
Also,
$$\sum_{ \substack{ \{u,v\}\subseteq V(G) \\  d_G(u,v)=1 }}  (\deg_G(u)+\deg_G(v))=
\sum_{uv\in E(G) } (\deg_G(u)+\deg_G(v))=M_1(G).$$
 It is well known that almost all graphs have diameter two. This means that graphs of diameter two play an important role in the theory of graphs.
%%%%%%%%%%%%%%%%%%%%%%%%%%%%%%%%%%
\begin{lm}  \label{Distance2}
If $G$ is a graph of diameter 2, then
\begin{eqnarray*}
\sum_{ \substack{ \{v_i,v_j\}\subseteq V(G) \\  d_G(v_i,v_j)=2 } }   (\deg_G(v_i)+\deg_G(v_j)) = 2(n-1)  |E(G)| - M_1(G).
\end{eqnarray*}
\end{lm}
\begin{proof}{
Since the diameter of $G$ is two and each vertex $v_i\in V(G)$ has $\deg_G(v_i)$ neighbours in $G$, the number of vertices in $V$ which their distance to $v_i$ is two equals $n-1-\deg_G(v_i)$. This implies that
\begin{eqnarray*}
\sum_{ \substack{ \{v_i,v_j\}\subseteq V(G) \\  d_G(v_i,v_j)=2 } }   (\deg_G(v_i)+\deg_G(v_j)) &=& \sum_{i=1}^n (n-1-\deg_G(v_i)) \deg_G(v_i)  \\
&=&  (n-1) \sum_{i=1}^n \deg_G(v_i) - \sum_{i=1}^n (\deg_G(v_i))^2  \\
&=& 2(n-1)  |E(G)| - M_1(G).
\end{eqnarray*}
}\end{proof}

%%%%%%%%%%%%%%%%%%%%%%%%%%%%%%%%% 
%%%%%%%%%%%%%%%%%%%%%%%%%%%%%%%%% 
%%%%%%%%%%%%%%%%%%%%%%%%%%%%%%%%%  DegreeDistance  or Schultz
%%%%%%%%%%%%%%%%%%%%%%%%%%%%%%%%% 
%%%%%%%%%%%%%%%%%%%%%%%%%%%%%%%%% 
\begin{theorem}  \label{DegreeDistance}
Let $G$ be an $n$-vertex graph of size $m$ whose diameter is 2.   If  $\mu$  is the Mycielskian of $G$, then the degree distance index of $\mu$ is given by
$$DD(\mu)=4 DD(G)- M_1(G) +( 7n-1)n + (8n+12)m,$$
where, $M_1(G)$ is the first Zagreb index of $G$.
\end{theorem}

\begin{proof}{
By the definition of degree distance index, we have
\begin{eqnarray*}
DD(\mu(G))&=&\sum_{\{ u,v \} \subseteq V(\mu)} d_\mu(u,v)~(\deg_\mu (u)+\deg_\mu(v)).
%\\ &=&  I_{x,x_i}+I_{x,v_i}+I_{x_i,x_j}+I_{\substack{x_i,x_j\\i\neq j}}+I_{x_i,v_i}+I_{v_i,v_j}
\end{eqnarray*}
Regarding to the different possible cases which $u$ and $v$ can be choosen from the set $V(\mu)$, we  consider the following cases.
We use the same notations as before. For computing degrees and distances two observations  \ref{MycielskiDegree} and   \ref{MycielskiDistance} are applied.
%\begin{eqnarray*}
%DD(\mu(G))=I_1+I_2+I_3+...+I_6
%\end{eqnarray*}
%%%
\\
{\large \bf Case 1.} $u=x$ and $v\in X$:
\begin{eqnarray*}
\sum_{i=1}^n d_\mu(x,x_i)~(\deg_\mu(x)+\deg_\mu(x_i))  = \sum_{i=1}^n 1~(n+(1+\deg_G(v_i))) = n(n+1)+2m.
\end{eqnarray*}
%%%
{\large \bf Case 2.} $u=x$ and $v\in V(G)$:
\begin{eqnarray*}
\sum_{i=1}^n  d_\mu(x,v_i) ~(\deg_\mu(x)+\deg_\mu(v_i)) = \sum_{i=1}^n 2~(n+2\deg_G(v_i)) = 2~(n^2+4m).
\end{eqnarray*}
%%%
{\large \bf Case 3.} $\{u,v\}\subseteq X$:
\begin{eqnarray*}
\sum_{\{x_i,x_j\}\subseteq X} d_\mu(x_i,x_j)~ (\deg_\mu(x_i)+\deg_\mu(x_j))  &=& \sum_{\{x_i,x_j\}\subseteq X} 2~ ((1+\deg_G(v_i))+(1+\deg_G(v_j)))  \\
&=& 2 \left( \sum_{\{x_i,x_j\}\subseteq X} 2 + \sum_{\{x_i,x_j\}\subseteq X} (\deg_G(v_i)+\deg_G(v_j))  \right)  \\
&=& 2 \left( 2\binom{n}{2} +  \sum_{k=1}^n (n-1) \deg(x_k) \right)  \\
&=& 2n^2-2n+4(n-1)m,
\end{eqnarray*}
Note that for each $x_k\in X$ we have $|\{ \{x_k,x_j\}:~j\neq k\}|=n-1$.
%%%
\\
{\large \bf Case 4.} $\{u,v\}\subseteq V(G)$:
\\
Since the diameter of $G$ is two,  Observation \ref{MycielskiDistance} implies that $d_\mu (v_i,v_j)=d_G(v_i,v_j)$ for each $v_i,v_j\in V(G)$. Hence,
\begin{eqnarray*}
\sum_{\{v_i,v_j\}\subseteq V(G)} d_\mu(v_i,v_j)~ (\deg_\mu(v_i)+\deg_\mu(v_j)) %&=& \sum_{\{v_i,v_j\}} d_\mu (v_i,v_j)~ (2\deg_G(v_i)+2\deg_G(v_j)) \\
&=&  \sum_{\{v_i,v_j\}\subseteq V(G)} d_G (v_i,v_j)~ (2\deg_G(v_i)+2\deg_G(v_j)) \\
&=& 2~ DD(G).
%\sum_{\substack{\{v_i,v_j\}\\ d_G (v_i,v_j) \leq 3}}  \deg_G(v_i,v_j) (\deg_G(v_i)+\deg_G(v_j))  \\
%&+&  2 \sum_{\substack{\{v_i,v_j\}\\ d_G (v_i,v_j)  \geq 4}}  4(\deg_G(v_i)+\deg_G(v_j)),
\end{eqnarray*}
%%%
\\
{\large \bf Case 5.} $u=v_i$ and $v=x_i$, $1\leq i\leq n$:
\begin{eqnarray*}
\sum_{i=1}^n ~d_\mu (v_i,x_i)~ (\deg_\mu(v_i)+\deg_\mu (x_i)) &=& \sum_{i=1}^n ~2~ ( 2\deg_G(v_i)+1+\deg_G(v_i) ) \\
&=& 2~(n+6m).
\end{eqnarray*}
%%%
\\
{\large \bf Case 6.} $u=v_i$ and $v=x_j$, $i\neq j$:
\begin{eqnarray*}
\sum_{\substack{\{v_i,x_j\} \subseteq V(\mu)  \\ i\neq j}} d_\mu(v_i,x_j)~(\deg_\mu(v_i)+\deg_\mu(x_j)) 
&=& \sum_{\substack{ \{v_i,x_j\} \subseteq V(\mu) \\ i\neq j}} d_\mu(v_i,x_j)~ (2\deg_G(v_i)+1+\deg_G(v_j) )  \\
&=& \sum_{\substack{ \{v_i,x_j\} \subseteq V(\mu) \\ i\neq j}} d_\mu(v_i,x_j)~ (\deg_G(v_i)+\deg_G(v_j))   \\
&+& \sum_{\substack{ \{v_i,x_j\} \subseteq V(\mu) \\ i\neq j } } d_\mu(v_i,x_j)~ (1+\deg_G(v_i)).
\end{eqnarray*}

Since $d_\mu(v_i,x_j)=d_\mu(v_j,x_i)$, and using Observation  \ref{MycielskiDistance}, we have
\begin{eqnarray*}
\sum_{\substack{ \{v_i,x_j\} \subseteq V(\mu) \\ i\neq j}} d_\mu(v_i,x_j)~ (\deg_G(v_i)+\deg_G(v_j)) &=&
2 \sum_{\substack{ \{v_i,v_j\} \subseteq V(G) \\ i\neq j }} d_\mu(v_i,x_j)~ (\deg_G(v_i)+\deg_G(v_j))  \\
&=& 2 \sum_{\substack{ \{v_i,v_j\} \subseteq V(G) \\ i\neq j }} d_G(v_i,v_j)~ (\deg_G(v_i)+\deg_G(v_j))  \\
&=& 2~DD(G).
%&+& 2 \sum_{\substack{ \{v_i,v_j\} \\ d_G (v_i,v_j) \geq 3 }}  3 ~ (\deg_G(v_i)+\deg_G(v_j)).
\end{eqnarray*}
Now since the diameter of $G$ is two,  Lemma \ref{Distance2} implies that
\begin{eqnarray*}
\sum_{\substack{ \{v_i,x_j\} \subseteq V(\mu) \\ i\neq j } } d_\mu(v_i,x_j)(1+\deg_G(v_i)) &=&
\sum_{\substack{ \{v_i,x_j\} \subseteq V(\mu)  \\ d_G(v_i,v_j)=1 }}  1~(1+\deg_G(v_i))  \\
&+&  \sum_{\substack{\{v_i,x_j\} \subseteq V(\mu) \\ d_G(v_i,v_j)=2 }}  2~(1+\deg_G(v_i))  \\
%&+&  \sum_{\substack{ (v_i,x_j) \\ d_G (v_i,v_j)  \geq 3 }}  3~(1+\deg_G(v_i))  \\
&=&  2 m + \sum_{v_iv_j \in E(G)} (\deg_G(v_i)+\deg_G(v_j)) \\
&+& 2 \left(  2 ({n\choose 2}-m)  +  \sum_{\substack{ \{v_i,v_j\} \subseteq V(G)\\ d_G(v_i,v_j)=2  }} (\deg_G(v_i)+\deg_G(v_j))\right) \\
%&+& \sum_{\substack{ (v_i,x_j) \\ d_G (v_i,v_j) \geq 3 }} 3~(1+\deg_G(v_i)) \\
&=& 2m+M_1(G) \\
&+&  2 ( n(n-1)-2m + 2(n-1)m-M_1(G)) \\
&=&  2n(n-1) + 2m(2n-3)-M_1(G).    
%&+& \sum_{\substack{ (v_i,x_j) \\ d_G (v_i,v_j) \geq 3 }} 3~(1+\deg_G(v_i)).
\end{eqnarray*}
Thus, 
\begin{eqnarray*}
\sum_{\substack{\{v_i,x_j\} \subseteq V(\mu) \\ i\neq j}} d_\mu(v_i,x_j)~(\deg_\mu(v_i)+\deg_\mu(x_j))  = 2~DD(G) - M_1(G) + 2n(n-1) + 2m(2n-3).   
\end{eqnarray*}
Therefore, using the cases 1 to 6, we obtain
\begin{eqnarray*}
DD(\mu)&=& \bigg{(}  n(n+1)+2m  \bigg{)}   +   \bigg{(}  2n^2+8m  \bigg{)}  +   \bigg{(}  2n^2-2n+4(n-1)m  \bigg{)} +  \bigg{(} 2~ DD(G)  \bigg{)}  \\ 
&+& \bigg{(}  2n+12m  \bigg{)}  +  \bigg{(}  2~DD(G) - M_1(G) + 2n(n-1) + 2m(2n-3)  \bigg{)}  \\
&=&  4~DD(G) - M_1(G) + (7n-1)n + (8n+12)m.
\end{eqnarray*}
}\end{proof}
%%%%%%%%%%%%%%%%%%%%%%%%%%%%%%%%%%
%%%%%%%%%%%%%%%%%%%%%%%%%%%%%%%%%%
%%%%%%%%%%%%%%%%%%%%%%%%%%%%%%%%%%
%%%%%%%%%%%%%%%%%%%%%%%%%%%%%%%%%%
%%%%%%%%%%%%%%%%%%%%%%%%%%%%%%%%%%
\section{The Randi$\acute{\mbox{c}}$ index}

In the following theorem we provide upper and lower bounds for the Randi$\acute{\mbox{c}}$ index of the Mycielskian graph.

\begin{theorem}
Let $G$ be an $n$-vertex graph of size $m$ with maximum degree $\Delta$ and minumum degree $\delta$, whose  Mycielskian graph is $\mu$. Then,
\begin{eqnarray*}
 {1\over 2} R(G) + {\sqrt{2}~m + \sqrt{n\Delta} \over \sqrt{\Delta^2+\Delta}}  
 \leq R(\mu) \leq 
 {1 \over 2} R(G) +  {\sqrt{2}~m + \sqrt{n\delta} \over \sqrt{\delta^2+\delta}}
\end{eqnarray*}
Moreover,  equalities hold if and only if $G$ is a regular graph.
\end{theorem}

\begin{proof}{
 By the definition of Randi$\acute{\mbox{c}}$ indx and by considering the various types of edges in $\mu$, one can see that
\begin{eqnarray*}
R(\mu) &=& \sum_{uv \in E(\mu)}  \frac{1}{\sqrt{\deg_\mu (u)~\deg_\mu (v)}}   \\
&=& \sum_{v_i v_j \in E(\mu)}  \frac{1}{\sqrt{\deg_\mu (v_i)~\deg_\mu (v_j)}}   \\
&+& \sum_{x_iv_j\in E(\mu)}  \frac{1}{\sqrt{\deg_\mu (x_i)~\deg_\mu (v_j)}}   \\
&+& \sum_{xx_i\in E(\mu)}  \frac{1}{\sqrt{\deg_\mu (x)~\deg_\mu (x_i)}}.   
\end{eqnarray*}

Since $N_\mu(x_i)=N_G(v_i)\cup \{x\}$ and using Observation \ref{MycielskiDegree} we get
\begin{eqnarray*}
 \sum_{x_iv_j\in E(\mu)}  \frac{1}{\sqrt{\deg_\mu (x_i)~\deg_\mu (v_j)}} &=&
 \sum_{x_iv_j\in E(\mu)}  \frac{1}{\sqrt{(1+\deg_G(v_i))~(2\deg_G(v_j))}}  \\
&=&      \sum_{i=1}^n \sum_{v_j \in N_G(v_i)} \frac{1}{\sqrt{(1+\deg_G(v_i))(2\deg_G(v_j))}}  \\
&\leq&   \sum_{i=1}^n \sum_{v_j \in N_G(v_i)} \frac{1}{\sqrt{(\delta +1)(2\delta)}}  \\
&=&       \sum_{i=1}^n \frac{\deg(v_i)}{\sqrt{(\delta+1)(2\delta)}}  \\
&=&       \frac{2m}{\sqrt{(\delta+1)(2\delta)}}  \\
&=&       \frac{\sqrt{2}m}{\sqrt{\delta^2+\delta}}.
\end{eqnarray*}
Similarly, we can see that 
\begin{eqnarray*}
 \frac{\sqrt{2}m}{\sqrt{\Delta^2+\Delta}} \leq \sum_{x_iv_j\in E(\mu)}  \frac{1}{\sqrt{\deg_\mu (x_i)~\deg_\mu (v_j)}}.
\end{eqnarray*}
Note that
\begin{eqnarray*}
\sum_{xx_i\in E(\mu)}  \frac{1}{\sqrt{\deg_\mu (x)~\deg_\mu (x_i)}} &=& 
\sum_{xx_i\in E(\mu)} \frac{1}{\sqrt{n(1+\deg_G(v_i))}}  \\  
&=&  \frac{1}{\sqrt{n}} \sum_{i=1}^n \frac{1}{\sqrt{1+\deg_G(v_i)}}.
\end{eqnarray*}
This implies that
$$ \sqrt{ \frac{n}{1+\Delta} }  \leq 
\sum_{xx_i\in E(\mu)}  \frac{1}{\sqrt{\deg_\mu (x)~\deg_\mu (x_i)}}
\leq \sqrt{ \frac{n}{1+\delta } } ~.$$
Now since 
\begin{eqnarray*}
  \sum_{v_i v_j \in E(\mu)}  \frac{1}{\sqrt{\deg_\mu (v_i)~\deg_\mu (v_j)}}
&=& \sum_{v_i v_j \in E(G)} \frac{1}{\sqrt{(2\deg_G(v_i))~(2\deg_G(v_j))}}  \\
&=& \frac{1}{2}R(G),
\end{eqnarray*}
the proof is complete.
}\end{proof}
%%%%%%%%%%%%%%%%%%%%%%%%%%%%%%%%%%
%%%%%%%%%%%%%%%%%%%%%%%%%%%%%%%%%%
%%%%%%%%%%%%%%%%%%%%%%%%%%%%%%%%%%
%%%%%%%%%%%%%%%%%%%%%%%%%%%%%%%%%%
%\section{Other indices!}
%?? "Narumi Katayama index    \cite{NarumiKatayamaIndex} :"  $NK(G)=\Pi_{x\in V(G)}\deg(x).$ ??

%In  \cite{Schultz-GutmanIndex} Gutman defined {\it the modified Schultz index}, which is now known as the {\it Gutman index}.
%The Gutman index of G, denoted by Gut(G), is defined as 
%$$Gut(G) = \sum_{\{u,v\}\subseteq V(G)} d(u,v)~\deg(u)\deg(v) $$

%%%%%%%%%%%%%%%%%%%%%%%%%%%%%%%%%%%%%%%%%%%%%%%%%%
%%%%%%%%%%%%%%%%%%%%%%%%%%%%%%%%%%%%%%%%%%%%%%%%%%
%%%%%%%%%%%%%%%%%%%%%%%%%%%%%%%%%%%%%%%%%%%%%%%%%%
%%%%%%%%%%%%%%%%%%%%%%%%%%%%%%%%%%%%%%%%%%%%%%%%%%

\end{document}